\documentclass[11pt]{amsart}

\usepackage[margin=1.25in]{geometry}
\usepackage[T1]{fontenc}
\usepackage{lmodern}
\usepackage{microtype}

\usepackage{amsmath,amssymb,amsthm,mathtools}
\numberwithin{equation}{section}

\usepackage[hidelinks]{hyperref}
\hypersetup{
pdfauthor={Xiang Fang},
pdftitle={Fourier Dimensions of Mandelbrot Cascades under Minimal Integrability}
}

\newtheorem{theorem}{Theorem}[section]
\newtheorem{proposition}[theorem]{Proposition}

\newtheorem{corollary}[theorem]{Corollary}

\theoremstyle{definition}
\newtheorem{definition}[theorem]{Definition}

\theoremstyle{remark}
\newtheorem{remark}[theorem]{Remark}

\newcommand{\E}{\mathbb{E}}
\newcommand{\Pbb}{\mathbb{P}}
\newcommand{\R}{\mathbb{R}}

\newcommand{\Leb}{\mathcal{L}}

\newcommand{\wh}{\widehat}

\newcommand{\dimF}{\dim_{\mathrm F}}
\newcommand{\dimE}{\dim_{\mathrm E}}
\newcommand{\dimtwo}{\dim_2}

\newcommand{\Aloc}{A_{\mathrm{loc}}}
\newcommand{\alphamin}{\alpha_{\min}}

\newcommand{\one}{\mathbf 1}
\newcommand{\spt}{\operatorname{spt}}

\title[Fourier dimensions of Mandelbrot cascades]
{Fourier Dimensions of Mandelbrot Cascades under Minimal Integrability}


\author[X. Fang]{Xiang Fang}
\address{Department of Applied Mathematics\\
National Yang Ming Chiao Tung University\\
Hsinchu 30010\\
Taiwan}
\email{xfang@nycu.edu.tw}

\begin{document}

\begin{abstract}
This note announces exact Fourier-dimension formulas for canonical Mandelbrot
cascade measures under the minimal Kahane--Peyri\`ere integrability condition,
and records the canonical \(b\)-adic extension on cubes.  The interval and
circle formulas are those of the joint work with Cai, Cheng, Li, Qu, and Xiao.
In the dyadic interval setting, the theorem is proved in a balanced
vector-weight model allowing dependence between sibling weights; almost surely
on non-extinction,
\[
    \dim_{\mathrm F}(\mu)
    =
    \dim_{\mathrm E}(\mu)
    =
    \dim_2(\mu)
    =
    D_{\mathrm E}(X).
\]
In particular, the scalar specialization gives the canonical
Mandelbrot--Kahane Fourier-dimension formula under the minimal integrability
condition.  On the circle, the endpoint formula is
\[
    \dim_{\mathrm F}(\mu_\circ)=A_{\mathrm{loc}}(W),
\]
where \(A_{\mathrm{loc}}(W)\) is the endpoint lower-local-dimension exponent.

The additional purpose of this note is to record the canonical \(b\)-adic
formula on cubes.  For the \(b\)-adic Mandelbrot cascade on
\([0,1]^d\subset\mathbb R^d\),
\[
    \dim_{\mathrm F}(\mu)
    =
    \min\{2,D_{\mathrm E}(W)\}
\]
almost surely on non-extinction.  The energy exponent is the natural
tree-cylinder square-sum exponent associated with the classical
scaling-exponent formalism.  The Fourier lower bound uses the flat dense-grid
recursion from the interval theorem, while the additional high-dimensional
upper obstruction is the universal Fourier barrier
\(\dim_{\mathrm F}(\mu)\le2\).
\end{abstract}

\maketitle

\section{Introduction}\label{s:Introduction}

\noindent
Mandelbrot introduced multiplicative cascades in connection with intermittency
in turbulence, following the phenomenological works of Kolmogorov,
Landau--Lifshitz, Obukhov, and Yaglom on turbulent energy dissipation
\cite{Kolmogorov1962,LandauLifshitz1959,Obukhov1962,Yaglom1966}.  These models
were developed in Mandelbrot's work on random multiplicative cascades
\cite{Mandelbrot1972,Mandelbrot1974a,Mandelbrot1976}.  The construction was
later placed on a rigorous martingale foundation by Kahane and Peyri\`ere
\cite{KahanePeyriere1976}.

In its simplest dyadic form, one starts with a nonnegative random variable \(W\)
satisfying the minimal Kahane--Peyri\`ere condition
\begin{equation}\label{eq:KP}
 W\ge0,\qquad \E W=1,\qquad
 \E[W\log_2^+W]<\infty,
 \qquad
 \E[W\log_2W]<1.
\end{equation}
Here \(\log_2^+x=\max\{\log_2 x,0\}\), and we use the convention
\(0\log_2 0=0\).  Independent copies of \(W\) are placed on the vertices of the
binary tree, and the level-\(n\) density is obtained by multiplying the weights
along each path.  The resulting measure-valued martingale converges weakly to a
finite random measure \(\mu\); the event \(\{\mu\ne0\}\) is the non-extinction
event.

For a finite Borel measure \(\nu\) on \(\R^m\), we use the Fourier transform
normalization
\[
        \wh{\nu}(\xi)
        =
        \int_{\R^m} e^{-2\pi i\langle x,\xi\rangle}\,d\nu(x).
\]
The question raised by Mandelbrot and later emphasized by Kahane is
harmonic-analytic: determine whether
\(
\wh\mu(\xi)\to0
\)
for \(|\xi|\to\infty\).
whether this convergence is polynomial, and what exponent is optimal
\cite{Kahane1993,Mandelbrot1976}.

Recent work has clarified this question in several regimes.  Chen--Li--Suomala
determined the Fourier dimension of Mandelbrot cascades under a subexponential
tail assumption \cite{ChenLiSuomala2025}.  Chen--Han--Qiu--Wang treated the
problem by a different method under all-positive-moment assumptions, together
with a nonlattice hypothesis \cite{ChenHanQiuWang2025Harmonic,ChenHanQiuWang2025}.
Lin--Qiu--Tan developed a general Fourier-dimension theory for classical
multiplicative chaos \cite{LinQiuTan2025}.  Ryou--Suomala obtained
Fourier-dimension results for cascades on planar curves under all-moment
assumptions \cite{RyouSuomala2026}.  The focus of the present note is the
minimal Kahane--Peyri\`ere regime \eqref{eq:KP}, including the infinite-variance
and pure heavy-tail cases allowed by that condition.

We shall use the following notation throughout this note.

\begin{definition}[Fourier, energy and correlation dimensions]
\label{def:dimensions}
Let \(\nu\) be a nonzero finite Borel measure on \(\R^m\).  The Fourier
dimension of \(\nu\) is
\begin{equation}\label{eq:Fourier-dim}
\dimF(\nu)
=
\sup\bigl\{0\le s\le m:
|\wh\nu(\xi)|=O(|\xi|^{-s/2}) \text{ as } |\xi|\to\infty
\bigr\}.
\end{equation}
The energy dimension of \(\nu\) is
\[
\dimE(\nu)=
\sup\biggl\{0<s<m:
\iint |x-y|^{-s}\,d\nu(x)d\nu(y)<\infty
\biggr\}.
\]

For the dyadic interval grid on \([0,1]\), let \(\mathcal D_n\) denote the
family of dyadic intervals of length \(2^{-n}\).  The corresponding correlation
dimension is
\begin{equation}\label{eq:dimtwo-dyadic}
        \dimtwo(\nu)
        =
        \liminf_{n\to\infty}
        \frac{-\log_2\sum_{I\in\mathcal D_n}\nu(I)^2}{n}.
\end{equation}
For the \(b\)-adic cube grid on \([0,1]^d\), let
\(\mathcal D_n^{b,d}\) denote the family of \(b\)-adic cubes of side length
\(b^{-n}\).  The corresponding correlation dimension is
\begin{equation}\label{eq:dimtwo-badic}
        \dimtwo(\nu)
        =
        \liminf_{n\to\infty}
        \frac{-\log_2\sum_{Q\in\mathcal D_n^{b,d}}\nu(Q)^2}
        {n\log_2 b}.
\end{equation}
\end{definition}

For the cascade measures considered below, the standard square-sum criterion
identifies the relevant correlation exponent with the energy dimension:
\[
        \dimE(\nu)=\dimtwo(\nu).
\]
We shall also use the deterministic inequality
\[
        \dimF(\nu)\le \dimE(\nu)
\]
for nonzero finite Borel measures; see, for instance, \cite{Mattila2015}.

The point of working under \eqref{eq:KP} is that no second-moment assumption is
imposed.  In particular, the minimal Kahane--Peyri\`ere condition allows the
heavy-tail regime
\[
        \E W^{1+\varepsilon}=\infty
        \qquad\text{for every }\varepsilon>0,
\]
where higher-moment concentration methods do not apply directly.  The interval
and circle formulas announced in Sections~2 and~3 are proved in the joint work
\cite{CaiChengFangLiQuXiao2026}.  On the interval, the obstruction is captured
exactly by the energy dimension; on the circle, the sharp obstruction is local
rather than global.  In particular, the scalar specialization gives the
canonical Mandelbrot--Kahane Fourier-dimension formula under the minimal
Kahane--Peyri\`ere integrability condition.

The square-sum energy exponent is closely related to the classical
scaling-exponent theory of independent random cascades, especially Molchan's
work on scaling exponents and multifractal dimensions
\cite{Molchan1996}.  In the dyadic vector notation of
\cite{CaiChengFangLiQuXiao2026}, the dictionary \(w_i=2X_i\) identifies
Molchan's exponent with \(\log_2\rho(q)\).  The role of
\cite{CaiChengFangLiQuXiao2026} is not to repackage this scaling formalism, but
to formulate the tree-cylinder energy statement in the exact form needed for
the Fourier proof and then to prove the matching Fourier lower bound under the
minimal Kahane--Peyri\`ere condition.

The additional purpose of the present note is to record the canonical
\(b\)-adic extension on cubes \([0,1]^d\subset\R^d\).  In that setting the
formula becomes
\[
        \dimF(\mu)=\min\{2,\dimE(\mu)\}
\]
almost surely on non-extinction.  The lower bound uses the same flat dense-grid
Fourier recursion as in the interval theorem.  The additional high-dimensional
upper obstruction is the universal Fourier barrier
\[
        \dimF(\mu)\le2.
\]

\section{Interval cascades under minimal integrability}
\label{s:interval-cascades}

We first recall the interval formula from \cite{CaiChengFangLiQuXiao2026}.  Its
natural form is slightly more general than the classical scalar cascade: the two
sibling weights may be dependent.  Let
\[
        X=(X_0,X_1)\in[0,\infty)^2
\]
be a random vector.  Independent copies of \(X\) are attached to the vertices
of the binary tree, while the two coordinates of a single copy are not assumed
to be independent.  We impose the balanced normalization
\begin{equation}\label{eq:balanced}
        \E X_0=\E X_1=\frac12,
\end{equation}
and the vector analogue of the minimal Kahane--Peyri\`ere condition,
\begin{equation}\label{eq:vector-KP}
\E[X_0\log_2^+X_0+X_1\log_2^+X_1]<\infty,
\qquad
\E[X_0\log_2X_0+X_1\log_2X_1]<0 .
\end{equation}
The balance condition is the Fourier-centering condition: it gives
\(\E\mu_n=\Leb\) at every level.

For a word \(u=(u_1,\ldots,u_n)\), define the path weight
\[
        L_u=\prod_{k=1}^n X_{u_k}(u|k-1),
\]
and the level-\(n\) cascade measure on \([0,1]\) by
\[
        d\mu_n(x)=\sum_{|u|=n}2^nL_u\one_{I_u}(x)\,dx .
\]
The weak limit is denoted by \(\mu\), and we write
\[
        M=\mu([0,1])
\]
for its total mass.  Under \eqref{eq:balanced} and \eqref{eq:vector-KP},
\(\mathbb{P}(M>0)>0\), and all dimension statements below are made on the
non-extinction event \(\{M>0\}\).

The relevant moment profile is
\begin{equation}\label{eq:rho}
        \rho(q)=\E(X_0^q+X_1^q),\qquad q>0.
\end{equation}
Define
\begin{equation}\label{eq:DE-vector}
        D_E(X)=
        \sup_{\substack{1<q<2\\ \rho(q)<1}}
        \left(-\frac2q\log_2\rho(q)\right),
\end{equation}
with value \(0\) if the set is empty.  This is the vector square-sum energy
exponent in the normalization used in \cite{CaiChengFangLiQuXiao2026}.

\begin{theorem}[Exact interval formula]
\label{thm:main-vector-exact-fourier-dimension}
Assume \eqref{eq:balanced} and \eqref{eq:vector-KP}.  Let \(\mu\) be the
limiting dyadic vector cascade measure on \([0,1]\), and set \(M=\mu([0,1])\).
Then, almost surely on \(\{M>0\}\),
\begin{equation}\label{eq:main-vector-exact-formula}
        \dimF(\mu)=\dimE(\mu)=\dimtwo(\mu)=D_E(X).
\end{equation}
Moreover, for every \(0<\sigma<D_E(X)\), there is a finite random constant
\(C_\sigma\) such that
\begin{equation}\label{eq:main-vector-decay}
        |\wh{\mu}(\xi)|
        \le
        C_\sigma |\xi|^{-\sigma/2}
        \qquad (|\xi|\gg1).
\end{equation}
\end{theorem}

The proof in \cite{CaiChengFangLiQuXiao2026} has two modules.  The first is the
energy module, which proves
\[
        \dimE(\mu)=\dimtwo(\mu)=D_E(X)
        \qquad\text{almost surely on } \{M>0\}.
\]
It is formulated in terms of tree-cylinder square sums, the form needed later
for the Fourier recursion.  The subcritical side uses fractional moments of
square sums.  The supercritical side uses a no-plateau lemma for the moment
profile, a finite-block weighted-growth obstruction, and an alive-tree
amplification argument.  This is the tree-cylinder version of the classical
scaling-exponent mechanism, adapted to the conditioning on non-extinction and
to the exact branching recursion used in the Fourier proof.

The second module is the Fourier module, which proves
\[
        \dimF(\mu)\ge D_E(X)
        \qquad\text{almost surely on } \{M>0\}.
\]
Here the balanced condition gives the centering \(\E\mu_n=\Leb\).  The Fourier
lower bound is then obtained directly in the vector-valued dyadic model, using
a vector-valued \(\ell^r\) contraction, a mesoscopic Fourier decomposition, and
a ladder iteration on dense frequency grids.  The deterministic inequality
\[
        \dimF(\nu)\le\dimE(\nu)
\]
for finite Borel measures gives the reverse Fourier inequality and completes
the exact formula.

\begin{corollary}[Scalar dyadic formula]
\label{cor:scalar-interval-formula}
For the usual scalar dyadic Mandelbrot cascade, let \(W_0,W_1\) be independent
copies of \(W\), and set
\[
        X_i=\frac{W_i}{2},\qquad i=0,1.
\]
Then
\[
        \rho(q)=2^{1-q}\E W^q .
\]
Under the scalar Kahane--Peyri\`ere assumptions \eqref{eq:KP}, almost surely on
\(\{\mu\ne0\}\),
\begin{equation}\label{eq:scalar-main}
        \dimF(\mu)=\dimE(\mu)=\dimtwo(\mu)=D^+(W),
\end{equation}
where
\begin{equation}\label{eq:Dplus}
D^+(W)=
\sup_{1<q<2}
\left[
2-\frac2q\bigl(1+\log_2\E W^q\bigr)
\right]_+,
\end{equation}
with the convention that the corresponding term is \(0\) when
\(\E W^q=\infty\).
\end{corollary}

Thus, in the canonical scalar setting, Corollary~\ref{cor:scalar-interval-formula}
gives the Mandelbrot--Kahane Fourier-dimension formula under the minimal
Kahane--Peyri\`ere integrability condition.  The scalar formula is obtained as
a specialization of the vector theorem, whose Fourier proof allows arbitrary
dependence between sibling weights.

\section{The circle endpoint formula}
\label{s:circle-endpoint}

The circle model is obtained by pushing the dyadic parameter cascade forward
under
\[
        f(t)=(\cos 2\pi t,\sin 2\pi t),
        \qquad 0\le t<1.
\]
Let \(\mu_\circ\) denote the limiting cascade on \(\mathbb S^1\).  Unlike in
the interval case, the sharp obstruction is not the global energy exponent but
a local endpoint exponent:
\begin{equation}\label{eq:Aloc}
        \Aloc(W)=
        \sup_{q>1}
        \left[\frac{q-1-\log_2\E W^q}{q}\right]_+,
\end{equation}
with the convention that the corresponding term is \(0\) when
\(\E W^q=\infty\).

There is a useful comparison with the interval exponent \(D^+(W)\) in
\eqref{eq:Dplus}.  Define
\[
        A_{<2}(W)=
        \sup_{1<q<2}
        \left[\frac{q-1-\log_2\E W^q}{q}\right]_+,
\]
with the same convention on infinite moments.  Then
\begin{equation}\label{eq:comparison}
        D^+(W)=2A_{<2}(W),
        \qquad
        A_{<2}(W)\le \Aloc(W).
\end{equation}
The factor \(2\) in the interval formula comes from the Fourier-dimension
normalization relative to the one-dimensional energy exponent.  On the circle,
the endpoint is instead governed by points where the cascade has the smallest
lower local dimension; this is the local obstruction created by curvature.

We now state the endpoint formula on the circle, the prototypical
curved-support case with nonzero curvature.  This is the circle theorem from
\cite{CaiChengFangLiQuXiao2026}.

\begin{theorem}[Endpoint formula on the circle]
\label{thm:circle}
Assume \eqref{eq:KP}.  Then, almost surely on the non-extinction event,
\begin{equation}\label{eq:circle-main}
        \dimF(\mu_\circ)=\Aloc(W).
\end{equation}
In particular, for every \(0<\sigma<\Aloc(W)\), there is a finite random
constant \(C_\sigma\) such that
\[
        |\wh{\mu_\circ}(\xi)|
        \le
        C_\sigma |\xi|^{-\sigma/2}
        \qquad (|\xi|\gg1,\;\xi\in\R^2).
\]
\end{theorem}

Theorem~\ref{thm:circle} is proved for the canonical scalar Mandelbrot cascade.
The vector-valued interval theorem, Theorem~\ref{thm:main-vector-exact-fourier-dimension},
does not presently have a circle analogue in \cite{CaiChengFangLiQuXiao2026}.
The extension from the circle to fixed \(C^2\) Jordan curves with nonvanishing
curvature is developed separately in \cite{CF-curve}.

The proof of Theorem~\ref{thm:circle} has two parts.  The upper bound is a
deterministic endpoint obstruction coming from local dimension.  The lower
bound is the main Fourier-analytic part, and its technical input is the
finite-\(r\) annular theorem stated below.

We first discuss the upper bound.  For a finite Borel measure \(\nu\) on
\(\mathbb S^1\), define the minimum lower local dimension
\[
\alphamin(\nu)
=
\inf_{x\in\spt\nu}
\liminf_{r\downarrow0}
\frac{\log_2\nu(B(x,r))}{\log_2 r}.
\]
The local-dimension theorem gives
\[
        \alphamin(\mu_\circ)=\Aloc(W)
\]
on non-extinction.  On the other hand, a deterministic curved-support estimate
gives
\[
        \dimF(\nu)\le \alphamin(\nu)
\]
for every nonzero finite measure supported on \(\mathbb S^1\).  Combining these
two estimates yields
\[
        \dimF(\mu_\circ)\le \Aloc(W).
\]

We next turn to the lower bound.  The finite-\(r\) annular theorem converts one
finite-moment witness into uniform Fourier decay on dyadic frequency annuli.

\begin{definition}[Finite-\(r\) witnessed hypothesis]
\label{def:finite-r-witnessed-hypothesis}
Let \(0<s<1\).  A nonnegative random variable \(U\) with \(\E U=1\) satisfies
the finite-\(r\) witnessed hypothesis at exponent \(s\) if there exist \(r>1\)
and \(\delta>0\) such that
\begin{equation}\label{eq:finite-r-witnessed-gap}
        \E[U^r]<\infty,
        \qquad
        2^{1-r}\E[U^r]\le 2^{-r(s+\delta)}.
\end{equation}
\end{definition}

Let \(\widetilde\nu\) be the dyadic scalar cascade on \([0,1)\) generated by
\(U\), and let
\(
\nu_\circ=f_\#\widetilde\nu
\)
be its pushforward to \(\mathbb S^1\).

\begin{theorem}[Finite-\(r\) annular theorem]
\label{thm:main-finite-r-annular}
Let \(0<s<1\), \(r>1\), and \(\delta>0\).  Let \(U\ge0\) satisfy
\(\E U=1\) and \eqref{eq:finite-r-witnessed-gap}.  Then there exist constants
\(C,c,\eta>0\), depending only on \(s,r,\delta\) and the law of \(U\), such
that, for every \(n\ge1\),
\begin{equation}\label{eq:finite-r-annular-estimate}
\Pbb\left(
\sup_{2^n\le |\xi|\le 2^{n+1}}
|\wh{\nu_\circ}(\xi)|>C2^{-sn/2}
\right)
\le
C\exp(-c2^{\eta n})+C2^{-cn}.
\end{equation}
Consequently,
\begin{equation}\label{eq:finite-r-annular-as-decay}
        |\wh{\nu_\circ}(\xi)|
        =
        O(|\xi|^{-s/2})
        \qquad (|\xi|\to\infty)
\end{equation}
almost surely.
\end{theorem}

The two terms in \eqref{eq:finite-r-annular-estimate} come from different parts
of the argument.  For the almost sure consequence, only the summability of the
right-hand side is used.  The theorem is formulated for an auxiliary weight
\(U\), since in the endpoint argument \(U\) will be chosen as a finite-moment
witness for a strict subendpoint exponent \(s<\Aloc(W)\).  This separates the
annular Fourier estimate from the endpoint optimization.

We emphasize two features of Theorem~\ref{thm:main-finite-r-annular}.  First,
\(U\) is not assumed bounded.  The proof uses predictable capping before
applying martingale concentration, and the cost of the cap is paid by an
\(r\)-tail compensator.  Second, the conclusion is annular and summable in
\(n\).  This is stronger than a pointwise-in-frequency estimate and is exactly
what is needed for almost sure Fourier decay.

The proof decomposes the phase \(t\mapsto \xi\cdot f(t)\) into a stationary
tube and dyadic derivative bands.  The stationary tube and very small derivative
bands are controlled by local-mass estimates.  On the remaining bands, the
Fourier integral is written as a sum of exact dyadic martingale arrays.  The
pre-bin part gains from oscillation before the natural phase-bin scale, while
the post-bin part is controlled by local-mass decay after that scale.
Predictable capping makes Freedman's inequality \cite{Freedman1975}
applicable without boundedness assumptions.  The resulting martingale estimates
have stretched-exponential tails, and the \(r\)-tail compensator is summably
small uniformly over the annulus.

We finally indicate how Theorem~\ref{thm:main-finite-r-annular} gives the lower
bound in Theorem~\ref{thm:circle}.  If \(0<s<\Aloc(W)\), then some
\(r>1\) witnesses \(s\), in the sense that
\[
        2^{1-r}\E[W^r]\le 2^{-r(s+\delta)}
\]
for some \(\delta>0\).  Applying Theorem~\ref{thm:main-finite-r-annular} with
\(U=W\) gives
\[
        |\wh{\mu_\circ}(\xi)|
        =
        O(|\xi|^{-s/2})
        \qquad (|\xi|\to\infty)
\]
almost surely on non-extinction.  Letting \(s\uparrow\Aloc(W)\) through a
countable sequence gives
\[
        \dimF(\mu_\circ)\ge \Aloc(W).
\]
Together with the upper bound, this proves \eqref{eq:circle-main}.  The stated
Fourier decay follows from the same argument for every strict subendpoint
exponent \(0<\sigma<\Aloc(W)\).

\section{The \texorpdfstring{\(b\)-adic}{b-adic} cube extension}
\label{s:badic-cube-extension}

We finally record the canonical \(b\)-adic extension on cubes.  This is the
additional consequence announced in this note.  The argument uses the same two
flat modules as the interval theorem: the tree-cylinder square-sum energy
module and the dense-grid Fourier-recursion module.  The curved endpoint module
from Section~\ref{s:circle-endpoint} is not involved.  The additional
high-dimensional obstruction is the ambient Fourier barrier
\[
        \dimF(\mu)\le 2.
\]

Let \(b\ge2\), \(d\ge1\), and set
\[
        B=b^d .
\]
Label the \(B\) first-generation \(b\)-adic subcubes of \([0,1]^d\) by
\(\{1,\ldots,B\}\), and write
\[
        T_B=\bigcup_{n\ge0}\{1,\ldots,B\}^n,
        \qquad
        \{1,\ldots,B\}^0=\{\varnothing\},
\]
for the rooted \(B\)-ary tree.  If \(u\in T_B\) and \(|u|=n\), let \(Q_u\) be
the corresponding \(b\)-adic cube of generation \(n\), and let \(a_u\) denote
its lower-left corner.  Thus
\[
        Q_u=a_u+b^{-n}[0,1]^d .
\]

Let
\[
        \{W_u:u\in T_B\setminus\{\varnothing\}\}
\]
be independent copies of a nonnegative random variable \(W\) satisfying the
\(b\)-adic Kahane--Peyri\`ere condition
\begin{equation}\label{eq:badic-KP}
        W\ge0,\qquad \E W=1,\qquad
        \E[W\log_2^+W]<\infty,\qquad
        \E[W\log_2W]<d\log_2 b .
\end{equation}
For \(u=(u_1,\ldots,u_n)\), set
\[
        Y_u=\prod_{k=1}^n W_{u|k},
        \qquad
        Y_\varnothing=1,
\]
and define the level-\(n\) measure on \([0,1]^d\) by
\[
        d\mu_n(x)
        =
        \sum_{|u|=n}Y_u\one_{Q_u}(x)\,dx .
\]
Under \eqref{eq:badic-KP}, the measures \((\mu_n)\) converge weakly almost
surely to a finite random measure \(\mu\), the canonical \(b\)-adic Mandelbrot
cascade on \([0,1]^d\).  All dimension statements below are made on the
non-extinction event
\(
 \{\mu\ne0\}.
\)

For \(0<s<d\) and \(0<q<2\), define
\begin{equation}\label{eq:badic-ms}
        m_s(q)
        =
        b^{\,d+\frac q2(s-2d)}\,\E[W^q],
\end{equation}
with the convention that \(m_s(q)=\infty\) if \(\E[W^q]=\infty\).  Set
\begin{equation}\label{eq:badic-DE-def}
        D_E(W)
        =
        \sup\Bigl\{
        s\in(0,d):\inf_{0<q<2}m_s(q)<1
        \Bigr\}.
\end{equation}
Equivalently,
\begin{equation}\label{eq:badic-DE-equivalent}
        D_E(W)
        =
        \sup_{q\in I(W)}
        \left(
        2d-\frac2q
        \left(
        d+\frac{\log_2\E[W^q]}{\log_2 b}
        \right)
        \right),
\end{equation}
where
\[
        I(W)=
        \{q\in(1,2):B^{1-q}\E[W^q]<1\},
\]
and the supremum over the empty set is \(0\).

\begin{theorem}[Energy theorem for the \(b\)-adic cube cascade]
\label{thm:badic-energy}
Let \(b\ge2\), \(d\ge1\), and let \(W\) satisfy \eqref{eq:badic-KP}.  Let
\(\mu\) be the associated canonical \(b\)-adic Mandelbrot cascade on
\([0,1]^d\).  Then, almost surely on \(\{\mu\ne0\}\),
\begin{equation}\label{eq:badic-energy-formula}
        \dimE(\mu)=\dimtwo(\mu)=D_E(W).
\end{equation}
\end{theorem}

\begin{proof}[Sketch of proof]
This is the tree-cylinder square-sum argument from the interval theorem with
the replacements
\[
        2\rightsquigarrow B=b^d,
        \qquad
        I_u\rightsquigarrow Q_u,
        \qquad
        2^{-n}\rightsquigarrow b^{-n}.
\]
Let
\[
        \Sigma_n=\sum_{Q\in\mathcal D_n^{b,d}}\mu(Q)^2,
        \qquad
        X_n^{(s)}=b^{sn}\Sigma_n .
\]
The branching relation for the \(b\)-adic cube cascade gives
\[
        X_{n+1}^{(s)}
        =
        \sum_{i=1}^B A_i^{(s)}X_n^{(i,s)},
        \qquad
        A_i^{(s)}=b^{s-2d}W_i^2,
\]
where \(X_n^{(1,s)},\ldots,X_n^{(B,s)}\) are independent descendant copies.
Thus the moment profile of the branching coefficients is
\[
        \E\sum_{i=1}^B \bigl(A_i^{(s)}\bigr)^{q/2}
        =
        b^{\,d+\frac q2(s-2d)}\E[W^q]
        =
        m_s(q).
\]
This is the \(b\)-adic tree-cylinder form of the classical scaling-exponent
mechanism for independent random cascades; compare Molchan's scaling-exponent
formalism \cite{Molchan1996}.  The point here is to use this exponent in the
precise tree-cylinder normalization needed for the Fourier-recursion argument.

The subcritical side is immediate from the same fractional-moment contraction
as in the interval proof.  If \(m_s(q)<1\) for some \(1<q<2\), then
\[
        b^{sn}\Sigma_n\to0
        \qquad\text{almost surely on }\{\mu\ne0\}.
\]
Consequently,
\[
        \liminf_{n\to\infty}
        \frac{-\log_2\Sigma_n}{n\log_2 b}
        \ge s
        \qquad\text{almost surely on }\{\mu\ne0\}.
\]
Letting \(s\uparrow D_E(W)\) gives
\[
        \dimtwo(\mu)\ge D_E(W).
\]

For the reverse inequality, the no-plateau argument, the finite-block
weighted-growth obstruction, and the alive-tree amplification from the interval
proof apply to the coefficients
\[
        A_i^{(s)}=b^{s-2d}W_i^2 .
\]
Thus, for every \(s>D_E(W)\),
\[
        \limsup_{n\to\infty} b^{sn}\Sigma_n=\infty
        \qquad\text{almost surely on }\{\mu\ne0\}.
\]
Hence
\[
        \dimtwo(\mu)\le s
\]
on non-extinction.  Letting \(s\downarrow D_E(W)\) gives
\[
        \dimtwo(\mu)=D_E(W).
\]
Finally, the deterministic equivalence between the Riesz energy dimension and
the \(b\)-adic square-sum exponent for finite measures on \([0,1]^d\) gives
\[
        \dimE(\mu)=\dimtwo(\mu)=D_E(W).
\]
\end{proof}

\begin{proposition}[Universal Fourier barrier for canonical cube cascades]
\label{prop:badic-fourier-barrier}
Let \(b\ge2\), \(d\ge1\), and let \(W\) satisfy \eqref{eq:badic-KP}.  Let
\(\mu\) be the associated canonical \(b\)-adic Mandelbrot cascade on
\([0,1]^d\).  Then, almost surely on \(\{\mu\ne0\}\),
\[
        \dimF(\mu)\le2.
\]
\end{proposition}

\begin{proof}[Sketch of proof]
It is enough to rule out Fourier decay faster than \(|\xi|^{-1}\) along a
coordinate axis.  Let \(\pi_1:[0,1]^d\to[0,1]\) be the first-coordinate
projection and set
\[
        \nu_1=(\pi_1)_{\#}\mu .
\]
Then
\[
        \wh{\mu}(Re_1)=\wh{\nu_1}(R),
        \qquad R\in\R.
\]
The projected measure \(\nu_1\) is itself a one-dimensional \(b\)-adic
Mandelbrot cascade.  Its first-generation weights are
\[
        V_j
        =
        b^{-(d-1)}
        \sum_{\ell=1}^{b^{d-1}} W_{j,\ell},
        \qquad
        j=1,\ldots,b,
\]
where the \(W_{j,\ell}\) are independent copies of \(W\).  The total mass of
\(\nu_1\) is the same as the total mass of \(\mu\), so \(\nu_1\) is nonzero on
the same non-extinction event.

The one-dimensional upper-bound mechanism of Chen--Li--Suomala
\cite{ChenLiSuomala2025} applies to this projected \(b\)-adic cascade; in the
minimal Kahane--Peyri\`ere regime the required martingale input is supplied by
nondegeneracy of the projected cascade.  It gives that, on non-extinction,
\[
        |\wh{\nu_1}(R)|
        \neq
        O(|R|^{-1-\varepsilon})
        \qquad (|R|\to\infty)
\]
for every \(\varepsilon>0\).  Since
\(\wh{\mu}(Re_1)=\wh{\nu_1}(R)\), no Fourier-dimension exponent strictly larger
than \(2\) can be admissible for \(\mu\).  Hence
\(
\dimF(\mu)\le2
\)
almost surely on \(\{\mu\ne0\}\).
\end{proof}

\begin{theorem}[High-dimensional canonical formula]
\label{thm:badic-fourier}
Let \(b\ge2\), \(d\ge1\), and let \(W\) satisfy \eqref{eq:badic-KP}.  Let
\(\mu\) be the associated canonical \(b\)-adic Mandelbrot cascade on
\([0,1]^d\).  Then, almost surely on \(\{\mu\ne0\}\),
\begin{equation}\label{eq:badic-main}
        \dimF(\mu)=\min\{2,D_E(W)\}.
\end{equation}
Equivalently,
\[
        \dimF(\mu)=\min\{2,\dimE(\mu)\}
        =
        \min\{2,\dimtwo(\mu)\}.
\]
\end{theorem}

\begin{proof}[Sketch of proof]
The lower bound is the flat Fourier module from the interval theorem.  We only
indicate the changes in the \(b\)-adic cube setting.  Replace the dyadic annulus
by the normalized \(b\)-adic shell
\[
        \mathcal A_{b,d}^{\infty}
        =
        \{\eta\in\R^d:1\le |\eta|_\infty\le b\}.
\]
For \(\eta\in\mathcal A_{b,d}^{\infty}\), write
\[
        X_n(\eta)=\wh{\mu}(b^n\eta),
        \qquad
        h_n(\eta)=\int_{[0,1]^d}e^{-2\pi i b^n\langle \eta,x\rangle}\,dx,
\]
and
\[
        F_n(\eta)=X_n(\eta)-h_n(\eta).
\]
The centering is unchanged:
\(
\E\mu_n=\Leb|_{[0,1]^d}.
\)
Moreover,
\[
        |h_n(\eta)|\lesssim b^{-n},
\]
and
\[
        |X_n(\eta)-X_n(\eta')|
        \lesssim
        \mu([0,1]^d)b^n|\eta-\eta'|_\infty .
\]
Thus the dense-grid passage from grid estimates to shell estimates is the same
as in the interval case.

The moment profile in the Fourier recursion is
\begin{equation}\label{eq:badic-kappa}
        \kappa(q)=B^{1-q}\E[W^q].
\end{equation}
For \(q\in I(W)\), set
\begin{equation}\label{eq:badic-beta}
        \beta_W(q)
        =
        -\frac{\log_b\kappa(q)}{q}
        =
        d-\frac1q
        \left(
        d+\frac{\log_2\E[W^q]}{\log_2 b}
        \right).
\end{equation}
Choose
\[
        0<\beta<\beta_0<\beta_1<\min\{1,\beta_W(q)\}.
\]
Let \(G_n\subset\mathcal A_{b,d}^{\infty}\) be a maximal
\(b^{-(1+\beta_1)n}\)-separated grid, and put
\[
        r_n=\lceil d(1+\beta_1)n\rceil .
\]
The factor \(d\) enters only through the cardinality estimate
\[
        \#G_n\asymp b^{d(1+\beta_1)n}.
\]

The mesoscopic decomposition has the same form as in the interval proof.  If
\(m=k-k'\), then
\[
\begin{aligned}
F_k(\eta)
&=
\sum_{|v|=m}
B^{-m}Y_v
e^{-2\pi i b^k\langle \eta,a_v\rangle}
F_{k'}^{(v)}(\eta)
\\
&\quad
+
h_{k'}(\eta)
\sum_{|v|=m}
B^{-m}(Y_v-1)
e^{-2\pi i b^k\langle \eta,a_v\rangle}.
\end{aligned}
\]
The first term is the centered descendant contribution, and the second is the
Lebesgue forcing term.  Applying the same vector-valued
\(\ell^{r_n}\)-contraction gives the ladder recursion
\[
        \Lambda_k^{(n)}
        \le
        C\left(
        r_n^{q/2}\kappa(q)^{\,k-k'}\Lambda_{k'}^{(n)}
        +
        b^{-qk'}
        \right),
        \qquad
        \Lambda_k^{(n)}
        =
        \E\bigl[\|F_k\|_{\ell^{r_n}(G_n)}^q\bigr].
\]
This is exactly the interval recursion with \(\rho(q)\) replaced by
\(\kappa(q)\).  The same logarithmic ladder iteration, followed by Markov's
inequality and Borel--Cantelli, yields
\[
        \sup_{\eta\in\mathcal A_{b,d}^{\infty}}
        |\wh{\mu}(b^n\eta)|
        \le
        C_{q,\beta}(\omega)b^{-\beta n}
\]
for all sufficiently large \(n\).  Equivalently,
\[
        |\wh{\mu}(\xi)|
        \le
        C_{q,\beta}(\omega)|\xi|^{-\beta}
        \qquad (|\xi|\gg1).
\]
Letting \(\beta\uparrow\min\{1,\beta_W(q)\}\), and then optimizing over
\(q\in I(W)\), gives
\[
        \dimF(\mu)\ge \min\{2,D_E(W)\}.
\]

For the upper bound, Theorem~\ref{thm:badic-energy} gives
\[
        \dimF(\mu)\le \dimE(\mu)=D_E(W).
\]
The additional ambient restriction is Proposition~\ref{prop:badic-fourier-barrier},
which gives
\(
\dimF(\mu)\le2.
\)
Therefore
\[
        \dimF(\mu)\le \min\{2,D_E(W)\}.
\]
Combining the lower and upper bounds proves \eqref{eq:badic-main}.  The
formulation with \(\dimE(\mu)\) and \(\dimtwo(\mu)\) follows from
Theorem~\ref{thm:badic-energy}.
\end{proof}

\begin{remark}
\label{rem:badic-truncation}
The truncation at \(2\) is a Fourier phenomenon, not an energy phenomenon.
When \(d\le2\), it is invisible because \(D_E(W)\le d\le2\).  In dimensions
\(d>2\), the energy dimension may exceed \(2\), but the canonical cascade cannot
have Fourier dimension larger than \(2\) because of the universal Fourier
barrier.
\end{remark}

\section*{Acknowledgments}

X. F. was partially supported by the National Science
and Technology Council, Taiwan, grant no. 114-2115-M-A49-003-MY3.

The authors used an artificial-intelligence tool for language editing and
\LaTeX\ formatting; all mathematical content was written, checked, and approved
by the authors.

\section*{Declaration of interests}

The authors do not work for, advise, own shares in, 
or receive funds from any organisation that could benefit from this article, 
and have declared no affiliation other than their research organisations.

\end{document}